\documentclass{article}

\date{October 28, 2025}

\usepackage{mathpazo}   

\usepackage{tikz}
\definecolor{myseagreen}{HTML}{3FBC9D}
\usepackage[color=myseagreen]{todonotes}

\usepackage{soul}
\usepackage{nameref}
\usepackage{mathtools}
\usepackage{empheq}
\usepackage{comment}
\usepackage[shortlabels,inline]{enumitem}

\setlist[enumerate]{nosep}

\usepackage[colorlinks=true,
linkcolor=refkey,
urlcolor=lblue,
citecolor=magenta]{hyperref}
\usepackage[doc,wmm,hhb]{optional}
\usepackage{xcolor}

\usepackage{float}
\usepackage{soul}
\usepackage{graphicx}
\definecolor{labelkey}{rgb}{0,0.08,0.45}
\definecolor{refkey}{rgb}{0,0.6,0.0}
\definecolor{Brown}{rgb}{0.45,0.0,0.05}
\definecolor{lime}{rgb}{0.00,0.8,0.0}
\definecolor{lblue}{rgb}{0.5,0.5,0.99}
\definecolor{OliveGreen}{rgb}{0,0.6,0}
\definecolor{tyrianpurple}{rgb}{0.4, 0.01, 0.24}

\colorlet{hlcyan}{cyan!30}

\usepackage{stmaryrd} 
\usepackage{amssymb}

\makeatletter
\def\namedlabel#1#2{\begingroup
   \def\@currentlabel{#2}%
   \label{#1}\endgroup
}
\makeatother

\newcommand{\seppthree}{\setlength{\itemsep}{-3pt}}

\usepackage[margin=1in,footskip=0.25in]{geometry}
\newcommand{\bM}{\ensuremath{{\mathbf{M}}}}

\newcommand{\bu}{\ensuremath{\mathbf{u}}}

\providecommand{\siff}{\Leftrightarrow}

\newcommand{\thalb}{\ensuremath{\tfrac{1}{2}}}
\newcommand{\menge}[2]{\big\{{#1}~\big |~{#2}\big\}}

\newcommand{\To}{\ensuremath{\rightrightarrows}}

\newcommand{\fenv}[1]%
{\ensuremath{\,\overrightarrow{\operatorname{env}}_{#1}}}
\newcommand{\benv}[1]%
{\ensuremath{\,\overleftarrow{\operatorname{env}}_{#1}}}

\newcommand{\pair}[2]{\left\langle{{#1},{#2}}\right\rangle}
\newcommand{\scal}[2]{\left\langle{#1},{#2}  \right\rangle}

\newcommand{\Tt}{\ensuremath{\widetilde{T}}}

\newcommand{\RR}{\ensuremath{\mathbb R}}

\newcommand{\RP}{\ensuremath{\mathbb{R}_+}}

\newcommand{\RM}{\ensuremath{\mathbb{R}_-}}

\newcommand{\RX}{\ensuremath{\,\left]-\infty,+\infty\right]}}

\newcommand{\dom}{\ensuremath{\operatorname{dom}}\,}

\DeclareMathOperator*{\argmin}{argmin}

\newcommand{\inte}{\ensuremath{\operatorname{int}}}

\newcommand{\ran}{\ensuremath{{\operatorname{ran}}\,}}
\newcommand{\rnk}{\ensuremath{{\operatorname{rank}}\,}}
\newcommand{\zer}{\ensuremath{\operatorname{zer}}}
\newcommand{\sol}{\ensuremath{\operatorname{sol}}}

\newcommand{\cspan}{\ensuremath{\overline{\operatorname{span}}\,}}

\newcommand{\Fix}{\ensuremath{\operatorname{Fix}}}
\newcommand{\Id}{\ensuremath{\operatorname{Id}}}

\newcommand{\pinf}{\ensuremath{+\infty}}

\newcommand{\bT}{\ensuremath{{\mathbf{T}}}}

\newcommand{\bS}{\ensuremath{{\mathbf{S}}}}
\newcommand{\bA}{\ensuremath{{\mathbf{A}}}}
\newcommand{\bK}{\ensuremath{{\mathbf{K}}}}
\newcommand{\bKT}{\ensuremath{{\mathbf{S}}}} 
\newcommand{\bZ}{\ensuremath{{\mathbf{Z}}}}

{\begin{list}{}{%
\settowidth{\labelwidth}{\textrm{#1~}}%
\setlength{\leftmargin}{\labelwidth+\labelsep}}}
{\end{list}}
\usepackage{amsthm}
\makeatletter
\def\th@plain{%
	\thm@notefont{}
	\itshape 
}
\def\th@definition{%
	\thm@notefont{}
	\normalfont 
}
\makeatother
\usepackage[capitalize,nameinlink]{cleveref}

\crefname{equation}{}{equations}
\crefname{chapter}{Appendix}{chapters}
\crefname{item}{}{items}
\crefname{enumi}{}{}

\newtheorem{theorem}{Theorem}[section]

\newtheorem{corollary}[theorem]{Corollary}

\newtheorem{proposition}[theorem]{Proposition}

\newtheorem{example}[theorem]{Example}

\newtheorem{fact}[theorem]{Fact}

\newtheorem{remark}[theorem]{Remark}

\providecommand{\gr}{\operatorname{gra}}
\providecommand{\gra}{\operatorname{gra}}

\providecommand{\spn}{\operatorname{span}}

\allowdisplaybreaks 

\usepackage{relsize}
\newcommand{\ift}{\ensuremath{\text{if }}}
\newcommand{\cl}[1]{\overline{#1}}
\newcommand{\crefpart}[2]{
  \hyperref[#2]{\namecref{#1}~\labelcref*{#1}~\ref*{#2}}
}

\author{
Heinz H.\ Bauschke\thanks{
Mathematics, University
of British Columbia,
Kelowna, B.C.\ V1V~1V7, Canada. E-mail:
\texttt{heinz.bauschke@ubc.ca}.},~
Walaa M.\ Moursi\thanks{
Department of Combinatorics and Optimization, 
University of Waterloo,
Waterloo, Ontario N2L~3G1, Canada.
  and
  Mansoura University, Faculty of Science,
  Mathematics Department,
  Mansoura 35516, Egypt.
 E-mail: \texttt{walaa.moursi@uwaterloo.ca}.}
,~and~
Shambhavi Singh\thanks{
Department of Combinatorics and Optimization, 
University of Waterloo,
Waterloo, Ontario N2L~3G1, Canada.
\texttt{shambhavi.singh@uwaterloo.ca}.}
}

\title{\textsf{
Eckstein-Ferris-Pennanen-Robinson duality revisited:\\
paramonotonicity, total Fenchel-Rockafellar duality,\\
and the Chambolle-Pock operator
}
}

\begin{document}

\maketitle

\begin{abstract}
Finding zeros of the sum of two maximally monotone operators involving a continuous linear operator is a central problem in optimization and monotone operator theory. We revisit the duality framework proposed by Eckstein, Ferris, 
Pennanen, and Robinson from a quarter of a century ago. 
Paramonotonicity is identified as a broad condition ensuring that 
saddle points coincide with the closed convex rectangle formed by the 
primal and dual solutions. Additionally, we characterize total duality in 
the subdifferential setting and derive projection formulas for sets 
that arise in the analysis of the Chambolle-Pock algorithm within 
the recent framework developed by Bredies, Chenchene, Lorenz, and Naldi. 
\end{abstract}
{ 


}

\section{Introduction}

Throughout this paper, we assume that 
\begin{equation}
\text{$X$ and $Y$ are real Hilbert spaces, and $L\colon X\to Y$ is continuous and linear.}
\end{equation}
We also assume that 
\begin{equation}
\text{$A$ and $B$ are maximally monotone on $X$ and $Y$, respectively.}
\end{equation}
(See, e.g., \cite{BC2017} for more on maximally monotone operators.)
A central problem in optimization and variational analysis is to 
\begin{equation}\label{e:p}
\text{Find }x\in X\text{ such that }0\in Ax+L^*BLx.
\end{equation}
We emphasize that we do not assume that $A+L^*BL$ or $L^*BL$ is maximally monotone. 
Following Eckstein and Ferris \cite{EF}, 
we encode the \emph{(primal) problem} \cref{e:p} by the triple 
\begin{equation}
(A,L,B). 
\end{equation}
The \emph{dual problem} is given by 
\begin{equation}
\label{e:d} 
\text{Find $y\in Y$ such that }
0\in B^{-1}y -LA^{-1}(-L^*y),
\end{equation}
which is now represented by the triple 
$(B^{-1},-L^*,A^{-1})$. 
This gives rise to the \emph{duality} operation
\begin{equation}
(A,L,B)^* := (B^{-1},-L^*,A^{-1})
\end{equation}
which features the pleasant \emph{biduality} (see \cite[Section~2.1]{EF})
\begin{equation}
(A,L,B)^{**} = (B^{-1},-L^*,A^{-1})^* = 
\big((A^{-1})^{-1},-(-L^*)^*,(B^{-1})^{-1}  \big)
=(A,L,B). 
\end{equation}
The set of solutions to the primal problem is given by 
\begin{subequations}
\begin{align}
\sol(A,L,B) &:= \menge{x\in X}{\text{$x$ solves the problem 
given by $(A,L,B)$}}\\
&= 
\menge{x\in X}{0\in Ax+L^*BLx}. 
\end{align}
\end{subequations}
For ease of notation, 
we abbreviate the set of the primal and the dual solutions by 
\begin{equation}
\label{e:Z}
Z := \sol(A,L,B) = \menge{x\in X}{0\in Ax+L^*BLx}
\end{equation}
and 
\begin{equation}
\label{e:K}
K := \sol(A,L,B)^* = \menge{y\in Y}{0\in B^{-1}y-LA^{-1}(-L^*y)}, 
\end{equation}
respectively. 
Closely related to these notions is the set of 
\emph{saddle points} 
{
(see \cref{r:saddle} below for the motivation of this name) 
or \emph{primal-dual solutions} 
} 
for $(A,L,B)$, given by 
\begin{equation}
\label{e:saddle}
\bKT := \menge{(x,y)\in X\times Y}{-L^*y\in Ax\land Lx\in B^{-1}y}. 
\end{equation}
The inclusion 
\begin{equation}
\label{e:SZK}
\bKT \subseteq Z\times K
\end{equation}
is always true  (see \cref{p:241120d} below). 
However, equality in \cref{e:SZK} may fail as 
Eckstein and Ferris (\cite[page~69]{EF}) observed: 

\begin{example}
\label{ex:skewbad}
Suppose that $X=Y=\RR^2$, $L=\Id$,
\begin{equation}
A = \begin{bmatrix} 
0 & -1 \\\
1 & 0 
\end{bmatrix},
\;\;\text{and}\;\;
B = \begin{bmatrix} 
0 & 1 \\\
-1 & 0 
\end{bmatrix}.
\end{equation}
Then $A^{-1}=B=-A$, $B^{-1}=A=-B$, and
we have the strict inclusion 
\begin{equation}
\bKT = \gra(-A) 
\subsetneqq \RR^2\times\RR^2 = Z\times K.
\end{equation}
\end{example}

Our goal in this paper is to carefully study the relationship 
between the primal problem \cref{e:p}, the dual problem \cref{e:d}, 
the corresponding solution sets \cref{e:Z}, \cref{e:K}, and the
set of saddle points \cref{e:saddle}. 
Our main results can be summarized as follows:

\begin{itemize}
\item We show that paramonotonicity of $A$ and $B$ is a quite general sufficient condition to guarantee that $\bS = Z\times K$ (\cref{c:para1}). 
\item We observe that --- in the subdifferential case --- the nonemptiness of $Z$
characterizes total Fenchel-Rockafellar duality (\cref{t:td}). 
\item We obtain formulas for projections onto sets arising in the 
Bredies-Chenchene-Lorenz-Naldi framework for studying the 
Chambolle-Pock algorithm (\cref{s:projsols}). 
\end{itemize}

Our work is complementary to that by 
Eckstein and Ferris \cite{EF} who focused on complementarity problems, 
by Pennanen \cite{Pennanen} who considered an even more general framework and derived criteria for maximal  monotonicity as did  Robinson \cite{Robinson}. 
We mention that our work extends previous work on 
\emph{Attouch-Th\'era duality}, 
i.e., when $Y=X$ and $L=\Id$  (see \cite{AT} and \cite{BBHM}). 
{
We also would like to mention 
references that were brought to our attention after 
the initial version of this manuscript 
was posted on arXiv: indeed,  
Dr.~Panos Patrinos pointed out related work 
on the Bredies-Chenchene-Lorenz-Naldi framework in 
\cite{LP}, \cite{EPLP}, and \cite{ELP};
and 
Dr.~Patrick Louis Combettes stressed the relevance 
of his work \cite{BAC} and \cite{theusual} for 
saddle point and product space formulations. 
}

The remainder of the paper is organized as follows.
In \cref{s:traverse}, we lay the foundation by analyzing the mappings 
which traverse between $Z$ and $K$, the sets of primal and dual solutions. 
The graph of these mappings is intimately connected to the set of saddle points as is observed in \cref{s:saddle}. 
The case when we have common zeros is characterized in 
\cref{s:commonzeros}. 
In \cref{s:para}, we reveal the set of saddle 
points $\bKT$ to be  
the convex ``rectangle'' $Z\times K$ 
whenever $A$ and $B$ are paramonotone. 
In \cref{s:subdiffs}, we show that the nonemptiness of $Z$ is 
equivalent to total Fenchel-Rockafellar duality. 
The Chambolle-Pock operator is revisited in 
the Bredies-Chenchene-Lorenz-Naldi framework in \cref{s:CP};
moreover, projection formula on associated sets are presented in 
\cref{s:projsols}. 
In \cref{s:normalcones}, we focus on the case when $A$ and $B$ are
normal cone operators which is useful for studying feasibility problem.
In the final \cref{s:prod}, we discuss a product space set up 
to deal with certain problems that may feature more than two operators. 

The notation we employ is fairly standard and largely follows 
\cite{BC2017}.

\section{Traversing between primal and dual solutions}

\label{s:traverse}

The following notation, which induces a way to traverse between 
primal and dual solutions, will be used extensively throughout the paper: 
Set 
\begin{equation}\label{eq:Kx}
(\forall x\in X)\quad 
K_x:=(-L^{-*}Ax)\cap (BLx)\subseteq Y
\end{equation}
and
\begin{equation}\label{eq:Zy}
(\forall y\in Y)\quad 
Z_y:=(L^{-1}B^{-1}y)\cap A^{-1}(-L^*y) \subseteq X, 
\end{equation}
{
where $L^{-*} = (L^{*})^{-1}$ and the inverse is 
in the sense of set-valued analysis.
}

\begin{proposition}
\label{p:241120b}
Let $(x,y)\in X\times Y$. 
Then the following hold:
\begin{enumerate}
\item 
\label{p:241120b1}
$y\in K_x$ 
$\siff$ $x\in Z_y$. 
\item 
\label{p:241120b1.5}
$L^*(K_x) = (-Ax)\cap L^*BLx$ and 
$L(Z_y)= (B^{-1}y)\cap LA^{-1}(-L^*y)$.
\item 
\label{p:241120b2}
$K_x\neq\varnothing$
$\siff$ $x\in Z$, and
$Z_y\neq\varnothing$
$\siff$ $y\in K$.
\item 
\label{p:241120b3}
$K_x\subseteq K$ and 
$Z_y \subseteq Z$. 
\item 
\label{p:241120b4}
$K_x$ and $Z_y$ are closed and convex. 
\end{enumerate}
\end{proposition}
\begin{proof}
We have the equivalences
\begin{subequations}
\begin{align}
y\in K_x
&\siff
y\in -L^{-*}Ax \land y \in BLx\\
&\siff
-L^*y \in Ax \land Lx \in B^{-1}y\\
&\siff
x \in A^{-1}(-L^*y) \land x \in L^{-1}B^{-1}y\\
&\siff
x\in Z_y,
\end{align}
\end{subequations}
and this proves \cref{p:241120b1}.

We now prove \cref{p:241120b1.5}:
First, suppose that $v\in K_x$, i.e., 
$v\in-L^{-*}Ax$ and $v\in BLx$.
It follows that $L^*v\in -Ax$ and $L^*v\in L^*BLx$, i.e., 
$L^*v \in (-Ax)\cap L^*BLx$. 
We've verified that $L^*(K_x)\subseteq (-Ax)\cap L^*BLx$. 
Conversely, suppose that $u\in (-Ax)\cap L^*BLx$. 
Then there exists $v\in BLx$ such that $u=L^*v$. 
Hence $-L^*v\in Ax$ and so $v\in -L^{-*}x$. 
Thus $v\in K_x$ and so $u = L^*v \in L^*(K_x)$. 
Altogether, we have shown that $L^*(K_x) = (-Ax)\cap L^*BLx$. 
The proof of the identity 
$L(Z_y)= (B^{-1}y)\cap LA^{-1}(-L^*y)$ is similar.

Turning to \cref{p:241120b2} and \cref{p:241120b3}, 
suppose that 
$K_x\neq\varnothing$, and let $k\in K_x$. 
It follows that 
\begin{equation}
k\in -L^{-*}Ax \land k \in BLx.
\end{equation}
First, this implies that 
$-L^*k \in Ax \land L^*k\in L^*BLx$, which upon adding, yields 
$x \in Z$. 
Second, this also implies 
$x \in A^{-1}(-L^*k) \land Lx \in B^{-1}k$. 
Hence $-Lx \in -LA^{-1}(-L^*k) \land Lx \in B^{-1}k$. 
Upon adding, we learn that $k\in K$. 
We have shown that $K_x \subseteq K$ and this is also trivially true 
if $K_x=\varnothing$. 

Now assume $x\in Z$. 
Then there exists $y\in BLx$ such that 
$-L^*y\in Ax$, i.e., $y\in -L^{-*}Ax$.
It follows that $y \in K_x$ and so $K_x\neq\varnothing$. 

The remaining statements in \cref{p:241120b2} and \cref{p:241120b3}
are proved similarly.

Because $A$ and $B$ are maximally monotone, 
the sets 
$Ax$ and $B(Lx)$ are convex and closed. 
The linearity and continuity of $-L^*$ then 
yields the convexity and closedness of $-L^{-*}Ax$.
Upon intersecting, we obtain  convexity and closedness $K_x$. 
The set $Z_k$ is treated similarly. 
\end{proof}

We now associate with \cref{eq:Kx} and \cref{eq:Zy} the following 
set-valued operators:
\begin{equation}
\bK\colon X\To Y\colon x\mapsto K_x
\;\;\text{and}\;\;
\bZ\colon Y\To X\colon y\mapsto Z_y.
\end{equation}

\begin{corollary}
\label{c:241120c}
We have 
\begin{equation}
\text{$\dom \bK = Z$ and $\ran \bK = K = \bigcup_{z\in Z}K_z$}
\end{equation}
as well as 
\begin{equation}
\text{
$\dom \bZ = K$ and $\ran \bZ = Z = \bigcup_{k\in K}Z_k$.
}
\end{equation}
\end{corollary}
\begin{proof}
The domain statements follow from \cref{p:241120b}\cref{p:241120b2}.
This and \cref{p:241120b}\cref{p:241120b3} 
imply $\ran\bK = \bigcup_{z\in Z}K_z\subseteq K$
and $\ran\bZ = \bigcup_{k\in K}Z_k\subseteq Z$. 
Now let $k\in K= \dom \bZ$. 
Then $Z_k\neq\varnothing$. 
Let $z\in Z_k$. 
By \cref{p:241120b}\cref{p:241120b1}, 
$k \in K_z \subseteq \ran \bK$.
Hence 
$K\subseteq\ran \bK$. 
Altogether, we deduce that 
$\ran \bK = K$. The range of $\bZ$ is treated similarly. 
\end{proof}

The next result states that the operators $L\bZ$ and $L^*\bK$ are skew. 

\begin{proposition}\label{p:241120a}
Suppose that $z_i\in Z$, and $k_i\in K_{z_i}$, for each $i\in\{0,1\}$. 
Then 
\begin{equation}
\label{e:241125c}
\pair{Lz_0-Lz_1}{k_0-k_1}
= \pair{z_0-z_1}{L^*k_0-L^*k_1}=0.
\end{equation}
\end{proposition}
\begin{proof}
The left equality in \cref{e:241125c} is trivial. 
Now, the assumption implies that each $k_i \in (-L^{-*}Az_i)\cap (BLz_i)$.
Thus each $-L^*k_i \in Az_i$ and $L^*k_i\in L^*BLz_i$.
Because $A$ and $L^*BL$ are monotone, we deduce that 
$0 \leq \scal{z_0-z_1}{(-L^*k_0)-(-L^*k_1)}$ 
and $0 \leq \scal{z_0-z_1}{L^*k_0-L^*k_1}$.
Altogther, 
$0=\scal{z_0-z_1}{L^*k_0-L^*k_1}$ which is the right equality in
\cref{e:241125c}.
\end{proof}

\section{Saddle points}

\label{s:saddle}

\begin{fact} 
\label{f:saddle}
The set of saddle points $\bKT$ is convex and closed. 
\end{fact}

\begin{remark}
\label{r:saddle}
\cref{f:saddle} can be proved by observing that the points in $\bKT$ are actually 
the zeros of the 
the maximally monotone operator on $X\times Y$, given by 
$(x,y)\mapsto (Ax\times B^{-1}y) + (L^*y,-Lx)$; 
see \cite[Proposition~2]{EF} (and also \cite[Proposition~26.33]{BC2017}). 
When $A=\partial f$ and $B = \partial g$, for proper 
lower semicontinuous convex functions $f$ and $g$ on 
$X$ and $Y$, respectively, then $\bKT$ is precisely the set of saddle point of the function
\begin{equation}
(x,y)\mapsto 
f(x)-g^*(y)+\scal{x}{L^*y}, 
\end{equation}
which is convex in $x$ and concave in $y$. 
This motivates the term ``saddle points'' for \cref{e:saddle}. 
One may also find the terminology ``Kuhn-Tucker vectors'' for $\bKT$ 
in the literature (see, e.g., \cite{Rockafellar}). 
\end{remark}

\begin{corollary}
\label{p:241120d}
The set 
\begin{equation}
\label{e:241125a}
\text{$\bKT = \gra \bK = \gra \bZ^{-1}$ is a closed convex subset of $Z\times K$.
}
\end{equation}
Moreover, the sets of primal and dual solutions 
\begin{equation}
\label{e:241125b}
\text{
$Z$ and $K$ are convex. 
}
\end{equation}
\end{corollary}
\begin{proof}
Taking $(x,y)\in X\times Y$, we get the equivalences
\begin{subequations}
\begin{align}
(x,y)\in \bKT
&\siff
y\in -L^{-*}Ax \land y \in B(Lx)\\
&\siff 
y \in K_x\\
&\siff (x,y)\in\gra \bK, 
\end{align}
\end{subequations}
which yields $\bKT = \gra  \bK$. 
On the other hand, \cref{p:241120b}\cref{p:241120b1} yields 
$\bK = \bZ^{-1}$.
Combining with \cref{f:saddle}, we obtain \cref{e:241125a}.

Because $\bKT$ is convex, so are the restrictions to 
the first and second components, i.e., $Z=\dom\bK$ and $K=\ran\bK$.
This proves \cref{e:241125b} and we are done. 
\end{proof}

\begin{remark} 
\label{r:241211a}
The fact that $Z$ and $K$ are convex is somewhere surprising 
because the operator sums occurring in their definition are 
not assumed to be maximally monotone.
\end{remark}

\section{Common zeros}

\label{s:commonzeros}

The next result characterizes when we have common zeroes for the problem 
\cref{e:p}:

\begin{proposition}[common zeros]\label{prop:commonzeros}
$\zer A\cap \zer L^*BL\neq \varnothing \siff K\cap \ker L^* \neq \varnothing$. 
\end{proposition}
\begin{proof} 
``$\Rightarrow$'':
Indeed, using \cref{p:241120b}\cref{p:241120b3}, 
\begin{subequations}
\begin{align}
\zer A\cap \zer L^*BL\neq \varnothing
&\siff 
(\exists z\in X)\;\; z\in\zer A\cap \zer L^*BL\\
&\siff 
(\exists z\in X)\;\; 
0\in Az\land 0\in L^*BLz\\
&\Rightarrow
(\exists z\in X)\;\; 
 L^{-*}(0)\subseteq L^{-*}Az 
\land (\exists y\in BLz)\;\; L^*y=0\\
&\siff 
(\exists z\in X)\;\; 
\pm\ker L^*\subseteq L^{-*}Az \land 
(\exists y\in Y)\;\; y\in BLz\cap \ker L^*\\
&\Rightarrow
(\exists y\in Y)\;\; y\in \ker L^* \cap BLz\cap (-L^{-*}Az)\\
&\siff 
(\exists y\in Y)\;\; y\in \ker L^* \cap K_z \subseteq \ker L^* \cap K. 
\end{align}
\end{subequations}

``$\Leftarrow$'': 
Conversely, assume that $K\cap \ker L^*\neq \varnothing$. 
Let $k\in K\cap \ker L^*$. 
By \cref{c:241120c}, $K=\bigcup_{z\in Z} K_z$. 
Hence there exists $z\in Z$ 
such that $k\in K_z\cap \ker L^*$. 
Thus
\begin{equation}
k \in (-L^{-*}Az) \cap (BLz) \cap \ker L^*. 
\end{equation}
Hence 
$-L^*k \in Az \land k \in BLz \land L^*k = 0$. 
This implies $0 = -0 = -L^*k \in Az \land 0 = L^*k \in L^*BLz$, 
i.e., $z\in \zer A \cap \zer L^*BL$. 
\end{proof}

\begin{proposition}[$0$ as a dual solution]
\label{prop:commonzeros2}
We have:  
$\zer L^{-*}A\cap \zer BL\neq \varnothing \siff 0\in K$. 
\end{proposition}
\begin{proof}
Indeed, using \cref{c:241120c}, we have 
\begin{subequations}
\begin{align}
\zer L^{-*}A\cap \zer BL\neq \varnothing
&\siff (\exists z\in X)\;\; 0\in (-L^{-*}Az) \cap  (BLz)\\
&\siff (\exists z\in X)\;\; 0\in K_z\\
&\siff 0\in K = \bigcup_{z\in Z} K_z, 
\end{align}
\end{subequations}
as claimed. 
\end{proof}

The previous two results generalize \cite[Proposition~2.10]{BBHM};
moreover, these generalization are different from each other 
as we now show:

\begin{example}  \label{ex:common0vs0inK}
Suppose that $Y=X$ and that $U$ is a proper closed linear subspace of $X$.
Further suppose that $A=P_U$ and $L=P_U=L^*$,
let $u^\perp\in U^\perp\smallsetminus\{0\}$, 
and suppose that 
$B\colon y\mapsto u^\perp$. 
Because $\ker L^*=U^\perp$, we have 
$(\forall x\in X)$ 
$x\in Z$
$\siff$
$0\in Ax+L^*BLx = P_Ux$
$\siff$
$x\in U^\perp$; thus, 
\begin{equation}
Z = U^\perp.
\end{equation}
It thus follows with \cref{c:241120c} that 
\begin{align}
K&=\bigcup_{z\in Z}K_z=\bigcup_{z\in U^\perp} (-L^{-*}Az)\cap (BLz)
= \bigcup_{z\in U^\perp} (-U^\perp\cap \{u^\perp\})=\{u^\perp\}.
\end{align}
Hence $0\notin K$, and, as predicted by 
\cref{prop:commonzeros2}, 
$\zer L^{-*}\cap\zer BL=\varnothing$. 
On the other hand, 
$K\cap \ker L^* = \{u^\perp\}\cap U^\perp = \{u^\perp\} \neq\varnothing$
and so 
\cref{prop:commonzeros} guarantees that 
$\zer A \cap \zer L^*BL\neq\varnothing$.
Indeed, $\zer A\cap \zer L^*BL = \zer P_U \cap \zer P_Uu^\perp = U^\perp$. 
\end{example}

\section{Paramonotonicity helps!}

\label{s:para}

Recall that $A:X\To X$ is \emph{paramonotone} if it is monotone and 
\begin{equation*}
\left. 
\begin{array}{c}
(x_0,u_0)\in\gr A\\
(x_1,u_1)\in\gr A\\
\scal{x_0-x_1}{u_0-u_1}=0
\end{array}
\right\}
\Rightarrow 
(x_0,u_1)\in\gr A \text{ and }(x_1,u_0)\in\gr A.
\end{equation*}
The term ``paramonotone'' was introduced in 
\cite{Paramonotone} for functions on $\RR^n$, 
and thoroughly studied in \cite{IusPara}. 
The class of paramonotone operators is large: it includes subdifferential 
operators, strictly monotone operators, displacement mappings, and others 
(see also \cite{BC2017} and \cite{BWY2014}).

\begin{theorem}\label{thm:paraZK}
Suppose that 
$A$ and $B$ are paramonotone. 
Then $(\forall z\in Z)$ $K_z=K$ and $(\forall k\in K)$ $Z_k=Z$.
\end{theorem}
\begin{proof}
Take $z_0,z_1$ in $Z$, and let each $k_i \in K_{z_i}$ 
By \cref{p:241120a}, 
\begin{equation}
\scal{Lz_0-Lz_1}{k_0-k_1} = 0 = \scal{z_0-z_1}{(-L^*k_0)-(-L^*k_1)}.
\end{equation}
On the other hand, the definition of $K_{z_i}$ yields
$(z_i,-L^*k_i)\in\gr A$ and $(Lz_i,k_i)\in\gr B$. 
Altogether, in view of the paramonotonicity of $A$ and $B$, 
we deduce that 
$(z_0,-L^*k_1)\in\gr A$, 
$(z_1,-L^*k_0)\in\gr A$, 
$(Lz_0,k_1)\in\gr B$, and 
$(Lz_1,k_0)\in\gr B$. 
Hence $k_1\in K_{z_0}$ and $k_0\in K_{z_1}$.
It follows that $K_{z_0}=K_{z_1}$ and thus $(\forall z\in Z)$ $K_z=K$.

The remaining assertion is proved similarly, after noting that an operator
is paramonotone if and only if its inverse is likewise. 
\end{proof}

\begin{remark}
\label{r:pararecovery}
Suppose that $A$ and $B$ are paramonotone, and that 
$k\in K$ is an \emph{arbitrary} dual solution. 
Then 
\cref{thm:paraZK} guarantees that the \emph{entire set of primal solutions}
is recovered via
\begin{equation}
Z = Z_k = \big(L^{-1}B^{-1}k\big) \cap \big(A^{-1}(-L^*k)\big).
\end{equation}
Similarly, a single primal solution $z\in Z$ recovers all 
dual solutions via $K = K_z = (-L^{-*}Az)\cap (BLz)$. 
Without paramonotonicity, this fails even when $Y=X$ and $L=\Id$ --- 
combine \cref{ex:skewbad} with \cref{p:241120d}. 
\end{remark}

\begin{corollary}\label{c:para1}
Suppose that $A$ and $B$ are paramonotone. 
Then the following hold: 
\begin{enumerate}
\item\label{c:para11} The sets of primal and dual 
solutions, $Z$ and $K$, are convex and closed.
\item\label{c:para12} 
$\gr \bK = \bKT$ and $\gr \bZ$ are the rectangles $Z\times K$ and $K\times Z$, respectively.
\item\label{c:para13b} $\scal{\cspan(Z-Z)}{\cspan(L^*K-L^*K)}=0=\scal{\cspan(LZ-LZ)}{\cspan(K-K)}$.
\end{enumerate}
\end{corollary}

\begin{proof}
\cref{c:para11}: 
Combine  \cref{thm:paraZK} with \cref{p:241120b}. 
\cref{c:para12}: 
Combine \cref{thm:paraZK} 
with \cref{p:241120d}.
\cref{c:para13b}: 
Combining \cref{thm:paraZK} 
with \cref{p:241120a}, we see that 
$\scal{Z-Z}{L^*K-L^*K}=0=\scal{LZ-LZ}{K-K}$ which yields the result.
\end{proof}

\begin{remark}
As in \cref{r:241211a}, 
it is surprising that paramonotonicity guarantees that $Z$ and $K$ are 
closed and convex even though the operator sums occurring in their definition
are not required to be maximally monotone. 
{
Moreover, it would be interesting to know whether 
it is possible to characterize 
the rectangle property in 
\cref{c:para1}\cref{c:para12}. 
}
\end{remark}

We learn more about the solution sets under additional assumptions:

\begin{proposition}
\label{p:para2} 
Suppose that $A$ and $B$ are paramonotone.
Then the following hold: 
\begin{enumerate}
\item\label{p:para21} 
If $\cspan(L^*K-L^*K)=X$, then $Z$ is a singleton. 
\item\label{p:para22} 
If $\cspan(LZ-LZ)=Y$, then $K$ is a singleton.
\item\label{p:para23} 
If $\cspan(K-K)=Y$, then $Z-Z\subseteq \ker L$.
\item\label{p:para24} 
If $\cspan(Z-Z)=X$, then $K-K\subseteq \ker L^*$.
\end{enumerate}
\end{proposition}
\begin{proof}
\cref{p:para21}: 
If $\cspan(L^*K-L^*K)=X$, 
then \cref{c:para1}\cref{c:para13b} yields $\cspan(Z-Z)=\{0\}$; 
equivalently, $Z$ is a singleton. 
\cref{p:para22}: 
The proof is similar to that of \cref{p:para21}. 
\cref{p:para23}: 
Suppose that $\cspan(K-K)=Y$. 
Then \cref{c:para1}\cref{c:para13b} yields $\{0\}=\cspan(LZ-LZ)$. 
It follows that 
$\{0\} = LZ-LZ = L(Z-Z)$; hence, $Z-Z\subseteq\ker L$. 
\cref{p:para24}: 
The proof is similar to that of \cref{p:para23}. 
\end{proof}

\begin{example}
\label{ex:241207a}
Suppose that $Y=X\neq\{0\}$, $A=0$, and $B=\Id$.
Clearly, $A$ and $B$ are paramonotone. 
Because $\ker(L^*L)=\ker L$,
we have 
\begin{equation}
\label{e:241207a1}
Z = \ker L = Z-Z.
\end{equation}
And because $(\ran L)^\perp = \ker L^*$ we deduce that 
\begin{equation}
\label{e:241207a2}
K = \{0\} = K-K. 
\end{equation}
We note the following:
\begin{enumerate}
\item If $L$ is injective, then $Z=\{0\}$, 
$\cspan(L^*K-L^*K)=\{0\}\neq X$, and 
$\cspan(LZ-LZ)=\{0\}\neq Y$; therefore,
the converse implications of 
\cref{p:para2}\cref{p:para21}\&\cref{p:para22} fail. 
\item 
The converse of \cref{p:para2}\cref{p:para23} fails. 
\item 
If $L\neq 0$, then $Z-Z=\ker L \subsetneqq X$ and 
so the converse of \cref{p:para2}\cref{p:para24} fails. 
\end{enumerate}
\end{example}

In the presence of a common-zero assumption, we can make further deductions:

\begin{theorem}
\label{t:241120e}
Suppose that $A$ and $B$ are paramonotone, and that 
$\zer L^{-*}A\cap\zer BL\neq\varnothing$. 
Then the following hold:
\begin{enumerate}
\item 
\label{t:241120e1}
$Z=(\zer A)\cap(L^{-1}\zer B)=(\zer A)\cap(\zer (BL))$ and $0\in K$.
\item 
\label{t:241120e2}
$\cspan L^* K\perp \cspan(Z-Z)$. 
\item 
\label{t:241120e3}
If $A$ is single-valued, then $K\subseteq \ker L^*$. 
\item 
\label{t:241120e4}
If $B$ is single-valued, then $K=\{0\}$.
\item 
\label{t:241120e5}
If $\inte Z\neq\varnothing$, then $K\subseteq \ker L^*$. 
\end{enumerate}
\end{theorem}
\begin{proof}
The common-zero assumption is equivalent to 
\begin{equation}
\label{e:241126a}
0\in K
\end{equation}
by \cref{prop:commonzeros2}. 

\cref{t:241120e1}:
In view of \cref{e:241126a}, we deduce from \cref{r:pararecovery} that  
$Z = (L^{-1}B^{-1}0) \cap (A^{-1}(-L^*0)) = \zer A \cap L^{-1}(\zer B) 
= \zer A \cap \zer(BL)$. 

\cref{t:241120e2}: 
By \cref{e:241126a}, 
$K = K-0 \subseteq K-K$.
Hence 
\cref{c:para1}\cref{c:para13b} yields the result. 

\cref{t:241120e3}: 
Assume that $A$ is single-valued. 
By \cref{t:241120e1}, $Z\subseteq\zer A$. 
Now let $z\in Z$. 
Then $Az=0$ and
\cref{thm:paraZK} yields 
\begin{equation}
K=K_z=(-L^{-*}Az)\cap (BLz)=(-L^{-*}0)\cap(BLz)=\ker L^*\cap (BLz) \subseteq \ker L^*. 
\end{equation}

\cref{t:241120e4}: 
Assume that $B$ is single-valued. 
From \cref{t:241120e1}, we have $Z\subseteq L^{-1}\zer B$, 
i.e., $LZ \subseteq\zer B$. 
Combining this with \cref{e:241126a} and
\cref{thm:paraZK}, we deduce that 
$(\forall z\in Z)$ 
\begin{equation}
\label{e:241120e}
\{0\}\subseteq K=K_z=(-L^{-*}Az)\cap (BLz)\subseteq(-L^{-*}Az)\cap(B\zer B)=
(-L^{-*}Az) \cap \{0\} \subseteq\{0\}. 
\end{equation}

\cref{t:241120e5}: 
Assume that $\inte Z\neq\varnothing$. 
Then $\spn(Z-Z) = X$. 
Thus \cref{p:para2}\cref{p:para24} yields 
$K-K \subseteq \ker L^*$.
On the other hand, 
$K = K-0 \subseteq K-K$ by \cref{e:241126a}.
Altogether, $K\subseteq \ker L^*$. 
\end{proof}

\section{The subdifferential case and Fenchel-Rockafellar duality}

\label{s:subdiffs}

Throughout this section, we assume that 
\begin{equation}
f\colon X\to\RX 
\;\; \text{and}\;\;
g\colon Y\to\RX 
\;\;\text{are convex, lower semicontinuous, and proper,}
\end{equation}
and that 
\begin{equation}
A = \partial f
\;\; \text{and}\;\;
B = \partial g. 
\end{equation}
The \emph{Fenchel-Rockafellar primal problem} asks to minimize
$f+g\circ L$. 
Standard subdifferential calculus and Fermat's rule yields the 
following result:
\begin{proposition}
\label{p:fgincl}
We have 
\begin{equation}
Z = \zer(\partial f + L^*\circ\partial g\circ L)
\subseteq 
\zer(\partial f+\partial(g\circ L))
\subseteq 
\zer\partial(f+g\circ L) = 
\argmin (f+g\circ L). 
\end{equation}
\end{proposition}

\begin{theorem}
\label{t:fg}
Suppose that $z\in Z$ and 
let $x\in \argmin(f+g\circ L)$.
Then $L^*(K_z)\subseteq L^*(K_x)$. 
\end{theorem}
\begin{proof}
Because of \cref{p:fgincl}, we have 
$f(z) + g(Lz) = f(x) + g(Lx) \in\RR$; hence, 
\begin{equation}
\label{e:241204a}
f(z)-f(x)=g(Lx)-g(Lz). 
\end{equation}
We also have $K_z\neq \varnothing$ and 
$\varnothing\neq L^*(K_z)= (-Az)\cap L^*BLz 
=(-\partial f(z))\cap L^*\partial g(Lz)$ by 
\cref{p:241120b}\cref{p:241120b2}\&\cref{p:241120b1.5}. 
Let $u\in (-\partial f(z))\cap L^*\partial g(Lz)$. 
Then there exists $v\in K_z$ such that 
$u = L^*v$. 
On the one hand, because $-u\in\partial f(z)$, we have 
$f(z)+\scal{x-z}{-u}\leq f(x)$ or 
\begin{equation}
\label{e:241204b}
f(z)-f(x)\leq \scal{x-z}{u}. 
\end{equation}
On the other hand, because $v\in K_z \subseteq \partial g(Lz)$, 
we have $g(Lz) + \scal{Lx-Lz}{v}\leq g(Lx)$ 
$\siff$ 
$g(Lz)-g(Lx) \leq -\scal{Lx-Lz}{v}=-\scal{x-z}{L^*v}=-\scal{x-z}{u}$ 
$\siff$ 
\begin{equation}
\label{e:241204c}
\scal{x-z}{u} \leq  g(Lx)-g(Lz). 
\end{equation}
Combining \cref{e:241204a}, \cref{e:241204b}, and \cref{e:241204c}, 
we deduce that 
\begin{equation}
\label{e:241204d}
f(z)-f(x)=\scal{x-z}{u} = \scal{Lx-Lz}{v}=g(Lx)-g(Lz). 
\end{equation}
Because $-u\in\partial f(z)$; equivalently, 
$\scal{z}{-u}= f(z)+f^*(-u)$, 
we learn from \cref{e:241204d} that 
$\scal{z}{-u}= (f(x)+\scal{x-z}{u})+f^*(-u)$. 
Thus $\scal{x}{-u} = f(x)+f^*(-u)$, i.e., 
\begin{equation}
\label{e:241204e}
-u \in \partial f(x). 
\end{equation}
And because $v\in\partial g(Lz)$, 
we have $\scal{Lz}{v} = g(Lz) + g^*(v)$ 
and we learn from \cref{e:241204d} that 
$\scal{Lz}{v} = (g(Lx)+\scal{Lz-Lx}{v}) + g^*(v)$. 
Thus $\scal{Lx}{v}=g(Lx) + g^*(v)$ and so 
$v \in \partial g(Lx)$ which further yields
\begin{equation}
\label{e:241204f}
u = L^*v  \in L^*\partial g(Lx). 
\end{equation}
Combining \cref{e:241204e}, \cref{e:241204f}, 
with \cref{p:241120b}\cref{p:241120b1.5}, 
we deduce that $u = L^*v \in L^*(K_x)$. 
\end{proof}

\begin{corollary}
\label{c:fg}
If $\zer(\partial f+L^*\circ \partial g\circ L)\neq\varnothing$, then 
\begin{equation}
\label{e:241204g}
Z = \zer(\partial f + L^*\circ\partial g\circ L)
= 
\zer(\partial f+\partial(g\circ L))
= 
\zer\partial(f+g\circ L) = 
\argmin (f+g\circ L). 
\end{equation}
\end{corollary}
\begin{proof}
Let $z\in \zer(\partial f+L^*\circ \partial g\circ L) = Z$ and
let $x\in \argmin (f+g\circ L)$. 
Because $z\in Z$, 
\cref{p:241120b}\cref{p:241120b2} yields
$K_z\neq\varnothing$. 
In view of \cref{t:fg}, $L^*(K_x)\neq\varnothing$. 
Hence $K_x\neq\varnothing$ and so $x \in Z$. 
Therefore, 
\begin{equation}
\argmin (f+g\circ L) \subseteq Z.
\end{equation}
Finally, the conclusion now follows from \cref{p:fgincl}. 
\end{proof}

\begin{remark}
Consider \cref{c:fg}. 
If 
$Z=\varnothing$, then \cref{e:241204g} may fail 
even when $Y=X$ and $L=\Id$; for more on this, 
see \cite[Chapter~13]{BMBook}. 
\end{remark}

We now involve the \emph{Fenchel-Rockafellar dual problem} which asks to 
minimize $g^* + f^*\circ(-\Id)$. 
The corresponding \emph{optimal values} of 
the primal and dual problems are 
\begin{equation}
\mu := \inf_{x\in X}\big(f(x)+g(Lx) \big)
\;\;\text{and}\;\;
\mu^* := \inf_{y\in Y}\big(g^*(y)+f^*(-L^*y) \big), 
\end{equation}
respectively. 
The Fenchel-Young inequality quickly yields
\begin{equation}
\label{e:mumu*}
\mu\geq -\mu^*.
\end{equation}

\begin{theorem}[total duality]
\label{t:td}
We have the equivalence
\begin{equation}
\label{e:241204h}
Z\neq\varnothing
\;\;\siff\;\;
\mu^*=-\mu\in \RR\; \text{and both infima defining $\mu,\mu^*$ are attained,}
\end{equation}
in which case 
$Z=\argmin(f+g\circ L)$. 
\end{theorem}
\begin{proof}
Either side of \cref{e:241204h} implies 
that $\dom g\cap L(\dom f)\neq\varnothing$.

``$\Rightarrow$'':
Suppose that $z\in Z$. 
Let $k\in K_z$. 
Then $(z,k)\in\bKT$ by \cref{p:241120d} and so 
$-L^*k\in \partial f(z)$ and $Lz \in \partial g^*(k)$. 
It thus follows from \cite[Theorem~19.1]{BC2017} that 
$z$ solves the primal problem, $k$ solves the dual problem, 
and $\mu^*=-\mu\in\RR$. 

``$\Leftarrow$'': 
Suppose 
$(x,y)\in X\times Y$ satisfies
$f(x)+g(Lx)=\mu=-\mu^*=-g^*(y)-f^*(-L^*y)$. 
By \cite[Theorem~19.1]{BC2017}, $(x,y)\in\bKT$. 
Hence \cref{p:241120d} and \cref{c:241120c} yield $x\in Z$ (and $y\in K$). 

Finally, the ``in which case'' statement is a consequence of \cref{c:fg}. 
\end{proof}

\begin{remark}
\cref{t:td} can be viewed as a variant of 
\cite[Corollary~5.3]{LFLL} which states that 
if $x$ is a primal solution and $\partial(f+g\circ L)(x) = \partial 
f(x) + L^*\partial g(Lx)$, then \emph{total duality} holds, 
i.e., $\mu^*=-\mu$ and both the primal and the dual Fenchel-Rockafellar 
problems have solutions. 
We refer the reader to the paper \cite{LFLL} by Li, Fang, L\'opez, and L\'opez for a
comprehensive study of total duality. 
\end{remark}

In contrast, it is possible that there is no duality gap ($\mu^*=-\mu$) but neither 
the primal nor the dual Fenchel-Rockafellar problem has a solution:

\begin{example}
Suppose that $X=Y=\RR^2$, $L=\Id$, and $f$ and $g$ are given by 
\begin{equation}
f(x_1,x_2) = \exp(x_1)+\exp^*(x_2)
\;\;\text{and}\;\;
g(x_1,x_2) = \exp(x_1)+\exp^*(-x_2).
\end{equation}
Then $\mu^* = -\mu = 0$ but neither Fenchel-Rockafellar problem has 
a solution.
\end{example}
\begin{proof}
Recall that 
\begin{equation}
\exp^*(\xi) = 
\begin{cases}
\pinf, &\text{if $\xi<0$;}\\
0, &\text{if $\xi=0$;}\\
\xi\ln(\xi)-\xi, &\text{if $\xi>0$.}
\end{cases}
\end{equation}
It follows that
\begin{align}
\mu &= 
\inf_{x_1\in\RR}(\exp(x_1)+\exp(x_1)) + 0 = 0,
\end{align}
but $\mu$ is not attained. 
Because $f^*(y_1,y_2)= \exp^*(y_1)+\exp(y_2) = f(y_2,y_1)$ and 
$g^*(y_1,y_2) = \exp^*(y_1)+\exp(-y_2) = g(-y_2,-y_1)$, 
it follows that 
The Fenchel-Rockafellar dual problem thus asks to minimize 
$(y_1,y_2)\mapsto g(-y_2,-y_1)+f(-y_2,-y_1)$; however,
again this problem has optimal value $\mu^*=0$ yet no solution. 
\end{proof}

We conclude this section with an excursion to the 
unconstrained LASSO problem. 

\begin{example}[LASSO]
Let $b\in\RR^m$, $\lambda>0$, and 
suppose that $X=\RR^n$, $Y=\RR^m$, and 
$L\in \RR^{m\times n}$.
Furthermore, 
we suppose that $f(x) = \lambda\|x\|_1 -\scal{x}{L^*b}$ 
and $g(y) = \thalb\|y\|^2 + \thalb\|b\|^2$.
Thus our maximally monotone operators are given by 
\begin{equation}
Ax = \partial f(x) = \lambda\partial \|x\|_1 - L^*b
\;\;\text{and}\;\;
By = \partial g(y)= \nabla g(y) = \Id.
\end{equation}
The value of the primal Fenchel-Rockafellar problem is 
\begin{equation}
\mu = \min_{x\in X} \big(f(x)+g(Lx)\big) = 
\min_{x\in X} \big(\thalb\|Lx-b\|^2 + \lambda\|x\|_1\big),
\end{equation}
which is the famous unconstrained LASSO problem \cite{Tib}. 
Note that we wrote ``$\min$'', not ``$\inf$'', which we justify now: 
indeed, 
the primal objective function $x\mapsto f(x)+g(Lx)$ is the sum of the 
nonnegative function $x\mapsto \thalb\|Lx-b\|^2$ and 
the coercive continuous function $x\mapsto\lambda\|x\|_1$, which 
guarantees that minimizers exist. 
Next, $f^*(x) = \iota_C((x+L^*b)/\lambda)$ and 
$g^*(y) = \thalb\|y\|^2-\thalb\|b\|^2$, where 
$C=[-1,1]^n$ is the unit ball with respect to the max norm. 
It follows analogously that the dual objective function 
$y\mapsto g^*(y)+f^*(-L^*y)$ has a minimizer; moreover, it is 
unique due to the strong convexity of $g^*$. Combining 
with the fact that $\dom g = \dom g^* = Y$, we have total duality and 
so $Z\neq\varnothing$ and $K\neq\varnothing$. 
Denote the unique dual solution by $k$:
\begin{equation}
K = \{k\}. 
\end{equation}
The inverses of $A$ and $B$ are given by $A^{-1}\colon x \mapsto  N_C((x+L^*b)/\lambda)$ and $B^{-1} = \Id$. 
If $x=(x_1,\ldots,x_n)$, then 
\begin{equation}
N_C(x) = N_{[-1,1]}(x_1)\times \cdots\times N_{[-1,1]}(x_n), 
\quad\text{where}\;\;
N_{[-1,1]}(\xi) = \begin{cases}
\RM, &\text{if $\xi=-1$;}\\
\{0\}, &\text{if $|\xi|<1$;}\\
\RP, &\text{if $\xi=1$;}\\
\varnothing, &\text{otherwise.}
\end{cases}
\end{equation}
Now 
$Z=Z_k = L^{-1}B^{-1}(k) \cap A^{-1}(-L^*k)$, 
which turns in our current setting into 
\begin{equation}
\label{e:lassoZ}
Z = L^{-1}(k) \cap N_C\big(\tfrac{1}{\lambda}L^*(b-k)\big). 
\end{equation}
We note that \cref{e:lassoZ} immediately yields the following:
\begin{enumerate}
\item If $L$ is injective ($\siff$ $\rnk(L)=n\leq m$), 
then $L^{-1}(k) = Z$ is a singleton.
\item If each $(L^*(b-k))_i \in \left]-\lambda,\lambda\right[$, i.e., 
$\tfrac{1}{\lambda}L^*(b-k)\in\inte C$, then $N_C(\tfrac{1}{\lambda}(L^*(b-k)))=\{0\}$ and so $Z=\{0\}$ is a singleton. 
\end{enumerate}
More general conditions ensuring that $Z$ is a singleton can be found in 
Tibshirani's \cite{Tibuni}; see also 
the recent paper by 
Berk, Brugiapaglia, and Hoheisel \cite{BBH}. 
\end{example}

\section{The Chambolle-Pock operator within the Bredies-Chenchene-Lorenz-Naldi framework}

\label{s:CP}

From now on we assume that 
\begin{equation}
\sigma>0,
\;
\tau > 0,
\; \text{and}\; 
\sigma\tau\|L\|^2 \leq 1.
\end{equation}
Recall that the \emph{Chambolle-Pock (a.k.a.\ Primal-Dual Hybrid Gradient) operator} for the problem \cref{e:p} is
defined by (see \cite{CP} and also \cite{CPsurvey})
\begin{equation} \label{e:T}
\bT \colon X\times Y \to X\times Y\colon 
\begin{bmatrix}x\\y\end{bmatrix} \mapsto 
  \begin{bmatrix}
  x^+\\y^+
  \end{bmatrix}
  :=\begin{bmatrix}
  J_{\sigma A}(x-\sigma L^*y)\\
  J_{\tau B^{-1}}\big({y+\tau L(2x^+-x)}\big)
  \end{bmatrix}. 
\end{equation}

\begin{proposition}\label{p:CPpd}
$\Fix \bT = \bKT$ is convex and closed. 
\end{proposition}
\begin{proof}
Let $(x,y)\in X\times Y$. Then 
\begin{subequations}
\begin{align}
(x,y)\in \Fix \bT
&\siff
\begin{bmatrix}
x\\y
\end{bmatrix}
=\bT\begin{bmatrix}
x\\y 
\end{bmatrix}\\
&\siff
x=J_{\sigma A}(x-\sigma L^*(y)) \land 
y=J_{\tau B^{-1}}(y+\tau Lx)\\
&\siff 
x-\sigma L^*y\in x+\sigma Ax \land 
y+\tau Lx\in y+\tau B^{-1}y\\
&\siff -L^*y\in Ax \land 
Lx\in B^{-1}y\\
&\siff (x,y)\in \bKT,
\end{align}
\end{subequations}
and so $\Fix\bT = \bKT$. 
Now invoke \cref{p:241120d}. 
\end{proof}

Following the recent framework proposed by 
Bredies, Chenchene, Lorenz, and Naldi \cite{BCLN}, 
we define block operators on $X\times Y$ via 
\begin{equation}
\bM := \begin{bmatrix}
\tfrac{1}{\sigma}\Id_X & -L^*\\
-L & \tfrac{1}{\tau}\Id_Y
\end{bmatrix} 
\;\;\text{and}\;\;
\bA := 
\begin{bmatrix}
A & L^* \\
-L & B^{-1}
\end{bmatrix}. 
\end{equation}
Then $\bA$ is maximally monotone with 
$\zer \bA = \Fix \bT$ while $\bM$ is a positive semidefinite preconditioner 
and 
\begin{equation}
\label{e:inv(A+M)}
(\bA +\bM)^{-1}
\colon \begin{bmatrix} x\\y\end{bmatrix}
\mapsto 
\begin{bmatrix}
J_{\sigma A}(\sigma x)\\
J_{\tau B^{-1}}\big(2\tau J_{\sigma A}(\sigma x)+\tau y  \big) 
\end{bmatrix}. 
\end{equation}
Moreover, 
\begin{equation}
\bT = (\bA +\bM)^{-1}\bM.
\end{equation}
The preconditioner $\bM$ induces a seminorm on $X\times Y$ via
$\|(x,y)\|_{\bM}^2 = \scal{(x,y)}{\bM(x,y)}$; it is a norm if and only 
if $\sigma\tau\|L\|^2<1$, in which case $\bT$ is firmly nonexpansive 
with respect to $\|\cdot\|_{\bM}$. 
We also assume that there exist another real Hilbert space $Z$ and 
a continuous linear operator $C\colon Z\to X\times Y$ such that 
\begin{equation}
CC^* = \bM.
\end{equation}
The associated reduced operator 
\begin{equation}
\label{e:Tt}
\Tt := C^*(\bA+\bM)^{-1}C, 
\end{equation}
which operates in $Z$, is (classically) firmly nonexpansive 
and 
\begin{equation}
\label{e:Fixx1}
\Fix \Tt = C^*(\Fix \bT) \;\;\text{is convex and closed.}
\end{equation}

We now discuss various cases. 

\subsection{The general case}
One way to find $C$ is by first constructing 
another operator $R$ in the following manner: 
Suppose that  $Z=X\times Y$ and $R\colon Y\to Y$ satisfies 
\begin{equation}
\label{e:Roccurs}
RR^* = \Id_Y - \sigma\tau LL^*
\;\;\text{and set}\;\;
C = \begin{bmatrix}
\tfrac{1}{\sqrt{\sigma}}\Id_X & 0 \\
-\sqrt{\sigma}L & \tfrac{1}{\sqrt{\tau}}R
\end{bmatrix}.
\end{equation}
For instance, $R$ can be chosen as the principal square root of 
$\Id_Y-\sigma\tau LL^*$ for which we even have $R=R^*$. 
Another possibility is to use a Cholesky factorization of $\Id_Y-\sigma\tau LL^*$ 
in the finite-dimensional case.

\subsection{The case when $K\subseteq \ker L^*$}

For simplicity, we shall assume that $Y$ is finite-dimensional, $R$ 
is chosen as the principal square root of $\Id_Y-\sigma\tau LL^*$, 
and $K \subseteq\ker L^*$. 
Then $(\forall k\in K)$ $R^2k = k$. 
If $k\in K\smallsetminus\{0\}$, then 
$k$ is an eigenvector of $R^2$ with eigenvalue $1$; 
thus, $k$ is also an eigenvector of $R$ with eigenvalue
$\sqrt{1}=1$. Hence 
\begin{equation}
(\forall k\in K)\quad  R
k = k,
\end{equation}
and this hold obviously when $k=0$. 

Now assume furthermore that $A$ and $B$ are paramonotone.
Combining \cref{p:CPpd} with \cref{c:para1}\cref{c:para11} 
yields $\Fix \bT = Z\times K$. 
Thus \cref{e:Fixx1} yields 
\begin{subequations}
\begin{align}
\Fix\Tt &= \menge{C^*(z,k)}{(z,k)\in Z\times K}
\\
&= 
\tfrac{1}{\sqrt{\sigma}}\menge{(z-\sigma L^*k,\sqrt{\sigma/\tau}k)}{(z,k)\in Z\times K} \\
&= \tfrac{1}{\sqrt{\sigma}} Z \times \tfrac{1}{\sqrt{\tau}}K.
\end{align}
\end{subequations}

\subsection{The scaled isometry case: $\sigma\tau LL^*=\Id_Y$}

\label{ss:scalediso}

In this subsection, we assume that 
\begin{equation}
\sigma\tau LL^* = \Id_Y.
\end{equation}
(Note that this corresponds to $R=0$ in \cref{e:Roccurs}.) 
In this case, we can pick $Z=X$ and 
\begin{equation}
\label{e:isoC}
C = \begin{bmatrix}
\tfrac{1}{\sqrt{\sigma}}\Id_X \\
-\sqrt{\sigma}L
\end{bmatrix}. 
\end{equation}
Assume first that 
$Y\neq \{0\}$. 
Then
\begin{equation}
1 = \|\Id_Y\| 
= \sigma\tau\|LL^*\|
= \sigma\tau\|L\|^2
= \sigma\tau\|L^*\|^2
\end{equation}
and 
$\|y\|^2 = \scal{y}{\sigma\tau LL^*y}
= \sigma\tau \scal{L^*y}{L^*y}=\sigma\tau\|L^*y\|^2
=\|L^*y\|^2/\|L^*\|^2$. 
Hence 
\begin{equation}
\text{$L^*$ is a constant-multiple of an isometry from $Y$ to $X$,}
\end{equation}
a statement that trivially holds when $Y=\{0\}$. 
If $L$ is a nonzero matrix, this means that rows of $L$ are pairwise orthogonal,
and each row vector has the same length $1/\sqrt{\sigma\tau}$. 
In any case, $L^*$ preserves the topology of $Y$.
Let $x\in X$ and let 
$V$ be a nonempty closed convex subset of $Y$. 
Then $L^*V$ is a nonempty closed convex subset of $X$ and 
\begin{equation}
P_{L^*V}(x) = L^*(P_VLx).
\end{equation}
More generally, if $\rho\in\RR\smallsetminus\{0\}$, then 
\begin{align}
\label{e:241124a}
P_{\rho L^*V}(x) 
&= P_{L^*(\rho V)}(x)
= L^*P_{\rho V}(Lx)
= L^*(\rho P_V(Lx/\rho))
= \rho L^*P_V(Lx/\rho). 
\end{align}

Now assume that $A$ and $B$ are paramonotone. 
Combining \cref{p:CPpd} with \cref{c:para1}\cref{c:para11} 
yields 
\begin{equation}
\label{e:Fixx6}
\Fix \bT = Z\times K. 
\end{equation}
Hence \cref{e:Fixx1} and \cref{e:isoC} result in 
\begin{equation}
\label{e:Fixx2}
\Fix \Tt = \tfrac{1}{\sqrt{\sigma}}(Z-\sigma L^*K). 
\end{equation}
Hence $Z-\sigma L^*K = \sqrt{\sigma} \Fix\Tt$ and so 
\begin{equation}
\label{e:Fixx4}
P_{Z-\sigma L^*K} 
= \sqrt{\sigma} P_{\Fix\Tt}\circ \tfrac{1}{\sqrt{\sigma}}\Id_X; 
\end{equation}
or, equivalently, 
\begin{equation}
\label{e:Fixx5}
P_{\Fix\Tt} = \tfrac{1}{\sqrt{\sigma}} P_{Z-\sigma L^*K}\circ \sqrt{\sigma}\Id_X.
\end{equation}

\begin{corollary}
\label{c:241127a}
Suppose that 
$\sigma\tau LL^*=\Id_Y$, and $A$ and $B$ are paramonotone. Then 
$(\forall \rho>0)$ $Z-\rho L^*K$ is closed and convex.
\end{corollary}
\begin{proof}
It follows from \cref{e:Fixx2} that 
$Z-\sigma L^*K$ is closed and convex.
But $\sigma>0$ has been chosen arbitrarily and $Z$ and $K$ do not depend on $\sigma$; consequently, the result follows. 
\end{proof}

\subsection{The Douglas-Rachford case: $Y=X$, $L=\Id_X$, and $\sigma=\tau=1$}

\label{ss:DR}

We specialize the scaled-isometry case (\cref{ss:scalediso}) further to 
$Y=X$ and $\sigma=\tau=1$. Then  
\begin{equation}
C \colon x\mapsto (x,-x)
\;\;\text{and}\;\;
C^* \colon (x,y)\mapsto x-y.
\end{equation}
The reduced operator $\Tt$ from \cref{e:Tt} turns into 
the Douglas-Rachford operator
\begin{equation}
x\mapsto x-J_Ax + J_BR_Ax,
\end{equation}
where $R_A:=2J_A-\Id$ is the reflectant of $A$. 
In the paramonotone case, 
\cref{e:Fixx2} turns into 
\begin{equation}
\label{e:Fixx3}
\Fix \Tt = Z-K.
\end{equation}

\section{Projections involving the solution sets}

\label{s:projsols}

Motivated by the descriptions of $\Fix\Tt$ in \cref{e:Fixx2} and \cref{e:Fixx3}, 
we study projections related to the solution sets $Z$ and $K$.
In this section we will assume that 
\begin{equation}
\text{$A$ and $B$ are paramonotone.}
\end{equation}

\subsection{The general paramonotone case}

\begin{theorem}\label{t:proj}
Suppose that $A$ and $B$ are paramonotone, 
let $(z_0,k_0)\in Z\times K$, 
$x\in X$, and $\rho\in\RR$. 
Then the following hold:
\begin{enumerate}
\item\label{t:proj1} 
$Z-\rho \overline{L^*K}$ is convex and closed.
\item\label{t:proj2} 
$P_{Z-\rho \cl{L^*K}}(x)=P_{Z}(x+\rho L^*k_0)+P_{-\rho \cl{L^*K}}(x-z_0)$.
\item\label{t:proj3} 
If $(Z-z_0)\perp \rho L^* K$, 
then $P_{Z-\rho \cl{L^*K}}(x)=P_Z(x)+P_{-\rho \cl{L^*K}}(x-z_0)$.
\item\label{t:proj4} 
If $Z\perp \rho (L^*K- L^*k_0)$, 
then $P_{Z-\rho \cl{L^*K}}(x)= 
P_Z(x+\rho L^* k_0)+P_{-\rho \cl{L^*K}}(x)$.
\end{enumerate}
\end{theorem}

\begin{proof}
Recall from \cref{c:para1}\cref{c:para11} that
$Z$ and $K$ are convex and closed. 
By \cref{c:para1}\cref{c:para13b}, 
$(Z-Z)\perp (\rho L^*K-\rho L^*K)$. 
It follows that 
\begin{equation}
(Z-z_0)\perp (-\rho \overline{L^*K}+\rho L^*k_0).
\end{equation}
By \cite[Proposition~29.6]{BC2017}, the set 
$Z-z_0-\rho\overline{L^*K}+\rho L^*k_0$ is closed and convex 
--- equivalently, $Z-\rho \overline{L^*K}$ is closed and convex 
which yields \cref{t:proj1} ---
and 
  \begin{equation}
    P_{Z-z_0-\rho \overline{L^*K}+\rho L^*k_0}
    =P_{Z-z_0}+P_{-\rho \cl{L^*K}+\rho L^*k_0}.
\end{equation}
This implies \cref{t:proj2} via 
\begin{subequations}
\begin{align}
P_{Z-\rho \cl{L^*K}}(x)
&=P_{(z_0-\rho L^*k_0)+Z-\rho \cl{L^*K}-(z_0-\rho L^*k_0)}(x)\\
&=(z_0-\rho L^*k_0)+P_{Z-\rho \cl{L^*K}-(z_0-\rho L^*k_0)}\big(x-(z_0-\rho L^*k_0)\big)\\
&=(z_0-\rho L^*k_0)+P_{Z-z_0}\big(x-(z_0-\rho L^*k_0)\big)
+P_{-\rho \cl{L^*K}+\rho L^*k_0}\big(x-(z_0-\rho L^*k_0)\big)\\
&=P_{Z}(x+\rho L^*k_0)+P_{-\rho \cl{L^*K}}(x-z_0).
\end{align}
\end{subequations}
Similarly, we deduce that \cref{t:proj3} via 
\begin{subequations}
\begin{align}
P_{Z-\rho \cl{L^*K}}(x)
&=  P_{z_0+(Z-z_0)-\rho \cl{L^*K}}(x)\\
&=z_0+P_{Z-z_0-\rho \cl{L^*K}}(x-z_0)\\
&=z_0+P_{Z-z_0}(x-z_0)+P_{-\rho \cl{L^*K}}(x-z_0)\\
&=P_{Z}(x)+P_{-\rho \cl{L^*K}}(x-z_0). 
\end{align}
\end{subequations}
Finally, \cref{t:proj4} follows from
\begin{subequations}
\begin{align}
P_{Z-\rho \cl{L^*K}}(x)
&=P_{-\rho L^*k_0+Z-\rho \cl{L^*K}+\rho L^*k_0}(x)\\
&=-\rho L^*k_0+P_{Z-\rho \cl{L^*K}+\rho L^*k_0}(x+\rho L^*k_0)\\
&=-\rho L^*k_0
   +P_{Z}(x+\rho L^*k_0)+P_{-\rho \cl{L^*K}+\rho L^*k_0}(x+\rho L^*k_0)\\
&=P_{Z}(x+\rho L^*k_0)+P_{-\rho \cl{L^*K}}(x).
\end{align}
\end{subequations}
The proof is complete. 
\end{proof}

\begin{theorem}\label{t:parasha}
Suppose that $A$ and $B$ are paramonotone, 
let $\rho>0$, and let $x\in X$. 
Then the following hold:
\begin{enumerate}
\item
\label{t:parasha1} 
If $k_0\in K$, then 
$J_{\rho A}P_{Z-\rho \cl{L^*K}}(x)=P_Z(x+\rho L^*k_0)$. 
\item
\label{t:parasha2} 
If $z_0 \in Z$ and  $(Z-z_0)\perp \rho L^*K$, 
then 
$J_{\rho A} P_{Z-\rho \cl{L^*K}}(x)=P_Z(x)$. 
\item 
\label{t:parasha3}
If $K\cap \ker L^*\neq \varnothing$, then 
$J_{\rho A} P_{Z-\rho \cl{L^*K}}(x)=P_Z(x)$. 
\end{enumerate}
\end{theorem}

\begin{proof}
Let  $z_0\in Z$. 
\cref{t:parasha1}: 
Let $k_0\in K$, and 
set $z:=P_Z(x+\rho L^*k_0)$. 
Using \cref{t:proj}\cref{t:proj2} and \cref{thm:paraZK}, 
we deduce that 
\begin{align}
P_{Z-\rho \cl{L^*K}}(x)-z&=P_{-\rho\cl{L^*K}}(x-z_0)
\in -\rho \cl{L^*K}=-\rho\cl{L^*K_z} \subseteq \rho\cl{Az}=\rho Az.
\end{align}
Hence 
$P_{Z-\rho \cl{L^*K}}(x)\in (\Id+\rho A)z$; 
equivalently, 
$z = J_{\rho A}P_{Z-\rho \cl{L^*K}}(x)$. 

\cref{t:parasha2}:
Assume that $(Z-z_0)\perp L^*K$, and set $z:=P_Z(x)$. 
Using \cref{t:proj}\cref{t:proj3} and \cref{thm:paraZK},
\begin{align}
P_{Z-\rho \cl{L^*K}}(x)-z&=
P_{-\rho \cl{L^*K}}(x-z_0)\in -\rho\cl{L^*K}=-\rho\cl{L^* K_z}\subseteq 
\rho\cl{Az}=\rho Az.
\end{align}
Hence 
$P_{Z-\rho \cl{L^*K}}(x)\in (\Id+\rho A)z$; 
equivalently, 
$z = J_{\rho A}P_{Z-\rho \cl{L^*K}}(x)$. 

\cref{t:parasha3}:
Assume that $k_0\in K\cap\ker L^*$. 
Then \cref{c:para1}\cref{c:para13b} yields
$0=\scal{Z-z_0}{L^*K-L^*k_0}=\scal{Z-z_0}{L^*K}$. 
Now apply \cref{t:parasha2}. 
\end{proof}

\begin{remark}
Specializing the results in this section to 
the Douglas-Rachford setting (see \cref{ss:DR}), we
recover the results of \cite[Section~6]{BBHM}. 
\end{remark}

\subsection{Revisiting the scaled isometry case}

In addition to the paramonotonicity assumption of this section, we assume in 
this subsection the scaled isometry case, i.e., 
\begin{equation}
\sigma\tau LL^* = \Id_Y. 
\end{equation}
We have seen in \cref{ss:scalediso} that $L^*K$ is already closed; 
therefore, we can remove all closure bars in \cref{t:proj} 
and \cref{t:parasha} in this case 
and obtain the following two results. 

\begin{corollary}\label{c:proj}
Suppose that $A$ and $B$ are paramonotone, 
and $\sigma\tau LL^* = \Id_Y$. 
Let $(z_0,k_0)\in Z\times K$, 
$x\in X$, and $\rho\in\RR$. 
Then the following hold:
\begin{enumerate}
\item\label{c:proj1} 
$Z-\rho{L^*K}$ is convex and closed.
\item\label{c:proj2} 
$P_{Z-\rho{L^*K}}(x)=P_{Z}(x+\rho L^*k_0)+P_{-\rho{L^*K}}(x-z_0)$.
\item\label{c:proj3} 
If $(Z-z_0)\perp \rho L^* K$, 
then $P_{Z-\rho{L^*K}}(x)=P_Z(x)+P_{-\rho{L^*K}}(x-z_0)$.
\item\label{c:proj4} 
If $Z\perp \rho (L^*K- L^*k_0)$, 
then $P_{Z-\rho{L^*K}}(x)= 
P_Z(x+\rho L^* k_0)+P_{-\rho{L^*K}}(x)$.
\end{enumerate}
\end{corollary}

\begin{corollary}\label{c:parasha}
Suppose that $A$ and $B$ are paramonotone, 
and $\sigma\tau LL^* = \Id_Y$. 
Let $\rho>0$ and $x\in X$. 
Then the following hold:
\begin{enumerate}
\item
\label{c:parasha1} 
If $k_0\in K$, then 
$J_{\rho A}P_{Z-\rho{L^*K}}(x)=P_Z(x+\rho L^*k_0)$. 
\item
\label{c:parasha2} 
If $z_0 \in Z$ and  $(Z-z_0)\perp \rho L^*K$, 
then 
$J_{\rho A} P_{Z-\rho{L^*K}}(x)=P_Z(x)$. 
\item 
\label{c:parasha3}
If $K\cap \ker L^*\neq \varnothing$, then 
$J_{\rho A} P_{Z-\rho{L^*K}}(x)=P_Z(x)$. 
\end{enumerate}
\end{corollary}

\begin{corollary}
\label{c:Gukeshwon}
Suppose that $A$ and $B$ are paramonotone, 
$\sigma\tau LL^* = \Id_Y$, and 
$K\cap\ker L^*\neq\varnothing$. 
Let $(x_0,y_0)\in X\times Y$.
Set 
$w_0 := \tfrac{1}{\sqrt{\sigma}}\big(x_0-\sigma L^*y_0\big)$.
Then 
\begin{align}
\label{e:241128c}
w &:= P_{\Fix \Tt}w_0 = 
\tfrac{1}{\sqrt{\sigma}} P_{Z-\sigma L^*K}(\sqrt{\sigma}w_0), 
\end{align}
while the $\bM$-projection of $(x_0,y_0)$ onto $\Fix\bT$ is given by 
\begin{equation}
\label{e:241128e}
\Big(P_Z(x_0-\sigma L^*y_0),J_{\tau B^{-1}}\big(2\tau L P_Z(x_0-\sigma L^*y_0)
{
-\tau LP_{Z-\sigma L^*K}(\sqrt{\sigma}w_0)
}
\big)\Big).
\end{equation}
\end{corollary}
\begin{proof}
Setting $\bu_0 := (x_0,y_0)$,
we see that $w_0 = C^*\bu_0$ using \cref{e:isoC}.
Then \cref{e:Fixx5} yields the right identity in \cref{e:241128c}. 
By \cite[Theorem~3.1(i)]{BMSW} and \cref{e:Fixx6}, we have 
\begin{equation}
\label{e:241128a}
\bu := P^{\bM}_{\Fix \bT}\bu_0=P^{\bM}_{Z\times K}\bu_0 = (\bM+\bA)^{-1}Cw; 
\end{equation}
here $P^{\bM}_{\Fix\bT}$ denotes the (well-defined!) 
projection onto $\Fix\bT$ with respect to  the $\bM$-norm. 
Combining \cref{e:241128a}, \cref{e:isoC}, and 
\cref{e:inv(A+M)} gives
\begin{equation}
\label{e:241128d}
\bu 
=
\begin{bmatrix}
J_{\sigma A}(\sqrt{\sigma}w)\\
J_{\tau B^{-1}}\big(2\tau LJ_{\sigma A}(\sqrt{\sigma}w)
{ -
\tau\sqrt{\sigma}Lw
}
\big)
\end{bmatrix}.
\end{equation}
In view of \cref{e:241128c}, we have, 
using now \cref{c:parasha}\cref{c:parasha3}, 
\begin{align}
J_{\sigma A}(\sqrt{\sigma}w)
&=
J_{\sigma A}P_{Z-\sigma L^*K}(\sqrt{\sigma} w_0)
=P_Z(\sqrt{\sigma} w_0)
=P_Z(x_0-\sigma L^*y_0). 
\end{align}
Substituting this into 
\cref{e:241128d} yields \cref{e:241128e}.
\end{proof}

\begin{remark}
\cref{c:Gukeshwon} is interesting in certain algorithmic setting 
(see \cite[Section~5]{BMSW}). The 
projection $P_Z(x_0-\sigma L^*
{
y_0
}
)$ occurs also in 
\cite[Lemma~5.1]{BMSW} where on the one hand $(A,B)$ is more restrictive 
(normal cones of linear subspaces instead of paramonotone operators)
and on the other hand $L$ is less restrictive (general instead of a scaled isometry).
\cref{c:Gukeshwon} also generalizes the Douglas-Rachford case 
in \cite[Section~5.1.2]{BMSW}. 
\end{remark}

\section{Examples involving normal cone operators}

\label{s:normalcones}

\subsection{Two normal cone operators: feasibility}

In this subsection, we assume, respectively,  that 
\begin{equation}
\label{e:feas1}
\text{$U$ and $V$ are nonempty closed convex subsets of $X$ and $Y$, and that 
$U \cap L^{-1}(V) \neq\varnothing$.}
\end{equation}
We also assume that 
\begin{equation}
\label{e:feas2}
A = N_U \;\;\text{and}\;\; B = N_V.
\end{equation}
Note that $A = \partial\iota_U$ and $B=\partial\iota_V$, 
which shows that we can also think about this as the set up in 
\cref{s:subdiffs}.

\begin{theorem}
\label{t:feas}
The following hold under assumptions \cref{e:feas1} and \cref{e:feas2}:
\begin{enumerate}
\item 
\label{t:feasZ}
$Z = U \cap L^{-1}(V)$. 
\item 
\label{t:feasK}
$K = N_{V-LU}(0) = N_{\cl{V-LU}}(0)$. 
\item 
\label{t:feascones}
If $U$ and $V$ are cones, then 
$K = V^\ominus \cap L^{-*}(U^\oplus)$. 
\item 
\label{t:feaslins}
If $U$ and $V$ are linear subspaces, then 
$K = V^\perp \cap L^{-*}(U^\perp)$. 
\item
\label{t:feasinte}
If $V\cap \inte LU\neq\varnothing$ or 
$\inte(V)\cap LU\neq\varnothing$ or $0\in\inte\cl{V-LU}$, then $K=\{0\}$. 
\end{enumerate}
\end{theorem}
\begin{proof}
\cref{t:feasZ}: 
Let $x\in X$.
We have the equivalences
$x\in Z$
$\siff$
$0\in Ax+L^*BLx$
$\siff$
$0\in N_U(x)+ L^*N_V(Lx)$
$\siff$
$x\in U \land Lx\in V$
$\siff$
$x\in U\cap L^{-1}(V)$. 

\cref{t:feasK}: 
Let $z\in Z$. 
By \cref{t:feasZ}, $z\in U\cap L^{-1}(V)$. 
Then 
\begin{subequations}
\begin{align}
K_z &= 
(-L^{-*}N_U(z))\cap N_V(Lz)
= (-N_{LU}(Lz))\cap N_V(Lz)\\
&=
N_{-LU}(-Lz)\cap N_V(Lz)
= N_{V-LU}(0),
\end{align}
\end{subequations}
where the last equality follows either from \cite[Proposition~16.61(i)]{BC2017} 
or by direct verification. Now apply \cref{c:241120c}. 

\cref{t:feascones}:
Using \cite[Propositions~6.27 and 6.37]{BC2017},
we obtain 
$N_{V-LU}(0)=(V-LU)^\ominus = (V+(-L)U)^\ominus
= V^\ominus \cap ((-L)U)^\ominus
= V^\ominus \cap ((-L)(U^\ominus)^\ominus)^\ominus
= V^\ominus \cap ((-L)^*)^{-1}(U^\ominus)
= V^\ominus \cap L^{-*}(U^\oplus)$.  

\cref{t:feaslins}: This follows from \cref{t:feascones}. 

\cref{t:feasinte}: The assumptions imply 
$T_{V-LU}(0)=Y$, hence $K=N_{V-LU}(0)=\{0\}$ by 
\cite[Proposition~6.44(ii)]{BC2017}. 
\end{proof}

\begin{remark}
\cref{t:feas}\cref{t:feaslins} appears also in \cite[Section~5.2]{BMSW}. 
In the Douglas-Rachford case ($Y=X$ and $L=\Id$), 
\cref{t:feasK} states that $K = N_{V-U}(0)$; 
see also \cite[Corollary~3.9]{31}. 
\end{remark}

\subsection{Normal cone operator of an affine subspace}

\begin{proposition}
\label{p:affpara}
Let $U$ be a closed affine subspace of $X$, 
suppose that $A= N_U$ and $B$ is paramonotone. 
Then the following hold:
\begin{enumerate}
\item 
\label{p:affpara1}
$Z=U\cap (L^{-1}B^{-1}L^{-*}((U-U)^\perp)=U\cap ( (L^*BL)^{-1}((U-U)^\perp))\subseteq U$.
\item 
\label{p:affpara2}
$(\forall z\in Z)\quad K=L^{-*}((U-U)^\perp)\cap (BLz)\subseteq L^{-*}((U-U)^\perp)$.
\item 
\label{p:affpara3}
$Z-Z\perp L^*K$.
\end{enumerate}
\end{proposition}
\begin{proof}
For every $x\in X$, we have 
\begin{equation}
Ax=\begin{cases}
(U-U)^\perp,&\ift x\in U;\\
\hfil\varnothing,&\ift x\notin U,
\end{cases}
\quad\;\text{and}\quad\;
A^{-1}x=\begin{cases}
U,&\ift x\in (U-U)^\perp;\\
\varnothing,&\text{otherwise}.
\end{cases}
\end{equation}

\cref{p:affpara1}:
Fix $x\in X$. 
From \cref{c:para1}\cref{c:para12}, we have the equivalences
\begin{subequations}
\begin{align}
x\in Z
&\siff (\exists y\in Y)\;\; 
Lx\in B^{-1}(y) \land -L^*y\in Ax \\
&\siff (\exists y\in Y)\;\; 
x\in L^{-1}B^{-1}(y) \land x\in U \land y\in -L^{-*}((U-U)^\perp)\\
&\siff (\exists y\in Y)\;\; 
x\in U \cap L^{-1}B^{-1}(y) \land y\in L^{-*}((U-U)^\perp)\\
&\siff 
x\in U\cap L^{-1}B^{-1}L^{-*}((U-U)^\perp),
\end{align}
\end{subequations}
which yields the conclusion. 

\cref{p:affpara2}: 
Let $z\in Z \subseteq U$ and $y\in Y$. 
By \cref{thm:paraZK}, $K_z=K$. 
Thus 
\begin{subequations}
\begin{align}
y\in K
&\siff y \in -L^{-*}Az \cap BLz\\
&\siff 
z \in U \land  y\in L^{-*}((U-U)^\perp) \cap BLz
\\
&\siff 
y\in L^{-*}((U-U)^\perp) \cap BLz, 
\end{align}
\end{subequations}
which yields the conclusion.

\cref{p:affpara3}: 
\cref{p:affpara1} and \cref{p:affpara2} yield
$Z\subseteq U$ and $K\subseteq L^{-*}((U-U)^\perp)$.
Hence 
$Z-Z\subseteq U-U$ and $L^*(K) \subseteq (U-U)^\perp$, 
and we are done. 
\end{proof}

\begin{remark} 
\cref{p:affpara}\cref{p:affpara3} was an assumption in various results 
considered earlier: 
\cref{t:proj}\cref{t:proj3}, 
\cref{t:parasha}\cref{t:parasha2},
\cref{c:proj}\cref{c:proj3}, and
\cref{c:parasha}\cref{c:parasha2}. 
\end{remark}

\section{Dealing with more than two operators via the product space}

\label{s:prod}

In this last section, we illustrate how problems that are more general 
than \cref{e:p} can be handled using a product space approach. 
To this end, we assume assume that 
\begin{equation}
\text{$Y_1,\ldots,Y_n$ are real Hilbert spaces, 
each $L_i\colon X\to Y_j$ is continuous and linear,}
\end{equation}
and 
\begin{equation}
\text{$B_1,\ldots,B_n$ are maximally monotone on $Y_1,\ldots,Y_n$,}
\end{equation}
respectively. 
The problem of interest is to 
\begin{equation}
\label{e:P}
\text{Find $x\in X\;\;$ such that $\;\;0\in Ax+ \sum_{j=1}^nL_j^*B_jL_jx$.}
\end{equation}
We utilize the usual product Hilbert space 
\begin{equation}
Y = Y_1\times \cdots \times Y_m,\;\;
L\colon X\to Y \colon x \mapsto (L_jx)_{j=1}^n,
\;\;\text{and}\;\;
B = B_1\times \cdots \times B_n. 
\end{equation}
Then 
\begin{equation}
x \;\; \text{solves}\;\;
0\in Ax + L^*BL x
\quad\siff\quad
x
\;\;\text{solves}\;\;\cref{e:P},
\end{equation}
which allows us to tap into the theory developed above. 
We will not translate all results to this framework; rather, 
we focus on some highlights. 
The dual problem of $(A,L,B)$ asks to 
find $y\in Y$ such that 
$0\in B^{-1}y -L A^{-1}(-L^*y)$, 
i.e., 
\begin{equation}
\label{e:D}
\text{Find $(y_1,\ldots,y_n)\in Y$ such that} 
\;\;(\forall i\in\{1,\ldots,m\})\;\;
0 \in B_i^{-1}(y_i)-L_iA^{-1}\big(-(L_1^*y_1+\cdots+L_m^*y_m)\big). 
\end{equation}

\begin{proposition}
Suppose that $A,B_1,\ldots,B_m$ are paramonotone and 
that $y\in Y$ solves \cref{e:D}. 
Then the entire set of primal solutions of \cref{e:P} is given by
\begin{equation}
Z = L^{-1}B^{-1}y \cap A^{-1}(-L^*y) = 
\bigcap_{i\in\{1,\ldots,m\}}L_i^{-1}B_i^{-1}y_i 
\cap A^{-1}\big(-(L_1^*y_1+\cdots+L_m^*y_m)\big). 
\end{equation}
\end{proposition}
\begin{proof}
The assumption implies that $B$ is paramonotone. 
Hence the result follows from  \cref{r:pararecovery}. 
\end{proof}

Next, we record a consequence of \cref{t:feas} in the current setting.

\begin{corollary}
\label{c:feasprod}
Let $U,V_1,\ldots,V_n$ be closed convex subsets of 
$X,Y_1,\ldots,Y_n$, respectively. 
Assume that $U \cap \bigcap_{j=1}^n L^{-1}(V_j)\neq\varnothing$, 
that $A=N_U$, and that each $B_j=N_{V_j}$. 
Set $V := V_1\times \cdots\times V_n$. 
Then the following hold: 
\begin{enumerate}
\item 
\label{c:feasprod1}
$Z = U \cap \bigcap_{j=1}^n L^{-1}(V_j)$. 
\item 
\label{c:feasprod2}
$K = N_{V-LU}(0)$. 
\item 
\label{c:feasprod3}
If $U$ and each $V_j$ are cones, then 
$K=(V_1^\ominus\times\cdots\times V_n^\ominus)
\cap L^{-*}(U^\oplus)$. 
\item 
\label{c:feasprod4}
If $U$ and each $V_j$ are linear subspaces, then 
$K=(V_1^\perp\times\cdots\times V_n^\perp)
\cap L^{-*}(U^\perp)$. 
\item 
\label{c:feasprod5}
If 
$U \cap \bigcap_{j=1}^n L_j^{-1}(\inte V_j)\neq\varnothing$, 
 then $K=\{0\}$. 
\end{enumerate}
\end{corollary}

Finally, to find primal and dual solutions, one may iterate the 
Chambolle-Pock operator $\bT$. 
In the present set up, it is given by the following formula: 

\begin{subequations}
\begin{align} \label{e:Tprod}
\bT \colon X\times (Y_1\times\cdots \times Y_n) 
&\to X\times (Y_1\times\cdots \times Y_n)\\
\begin{bmatrix}x\\y_1\\\vdots\\y_n\end{bmatrix} 
&\mapsto 
  \begin{bmatrix}
  x^+\\y_1^+\\\vdots\\y_n^+
  \end{bmatrix}
  =\begin{bmatrix}
  J_{\sigma A}\big(x-\sigma \sum_{j=1}^n L_j^*y_j\big)\\
  J_{\tau B_1^{-1}}\big(y_1+\tau L_1(2x^+-x)\big)\\
  \vdots\\
  J_{\tau B_n^{-1}}\big(y_n+\tau L_n(2x^+-x)\big)\\
  \end{bmatrix}. 
\end{align}
\end{subequations}

Of course, other splitting methods can be applied to solve \cref{e:P}. 
We refer the reader to the nice recent book by 
Ryu and Yin \cite{RyuYin} and the nice recent survey 
by Condat, Kitahara, Contreras, and Hirabayashi \cite{Condat}. 

\section*{Acknowledgments}
{ The authors thank two reviewers and the associate
editor 
for their thoughtful and constructive remarks.}
The research of HHB and WMM was supported by Discovery Grants from 
the Natural Sciences and Engineering Research Council of Canada. 
The research of SS was supported by a postdoctoral fellowship from the Natural Sciences and Engineering Research Council of Canada.


\begin{thebibliography}{999}
\seppthree


\bibitem{AT}
H.\ Attouch,  M.\ Th{\' e}ra:
A general duality principle for the sum of two operators,
\emph{Journal of Convex Analysis}~3 (1996): 1--24.
\url{https://www.heldermann.de/JCA/jca03.htm}

\bibitem{BBHM}
H.H.\ Bauschke, R.I.\ Bo\c{t}, W.L.\ Hare, W.M.\ Moursi:
Attouch-Th\'era duality revisited:
paramonotonicity and operator splitting,
\emph{Journal of Approximation Theory}~164 (2012), 1065--1084. 
\url{https://doi.org/10.1016/j.jat.2012.05.008}

\bibitem{BC2017}
H.H.\ Bauschke, P.L.\ Combettes: 
\emph{Convex Analysis and Monotone Operator Theory in Hilbert Spaces},
2nd edition, Springer, 2017.
\url{https://doi.org/10.1007/978-3-319-48311-5}


\bibitem{31}
H.H.\ Bauschke, P.L.\ Combettes, D.R.\ Luke: 
Finding best approximation pairs relative to two closed convex sets in Hilbert spaces, 
\emph{Journal of Approximation Theory}~127 (2005), 178--192.
\url{https://doi.org/10.1016/j.jat.2004.02.006}


\bibitem{BMBook}
H.H.\ Bauschke, W.M.\ Moursi:
\emph{An Introduction to Convexity, Optimization, and Algorithms},
Society for Industrial and Applied Mathematics, 2023.
\url{https://doi.org/10.1137/1.9781611977806}


\bibitem{BMSW}
H.H.\ Bauschke, W.M.\ Moursi, S.\ Singh, X.\ Wang:
On the Bredies-Chenchene-Lorenz-Naldi algorithm: linear relations and strong convergence, 
\emph{SIAM Journal on Optimization}~35 (2025), 1963--1992. 
\url{https://doi.org/10.1137/23M1587919}

\bibitem{BWY2014}
H.H.\ Bauschke, X.\ Wang, L.\ Yao:
Rectangularity and paramonotonicity of maximally monotone operators, 
\emph{Optimization}~63 (2014), 487--504.
\url{https://doi.org/10.1080/02331934.2012.707653}

\bibitem{BBH}
A.\ Berk, S.\ Brugiapaglia, T.\ Hoheisel:
LASSO reloaded: a variational analysis perspective 
with applications to compressed sensing,
\emph{SIAM Journal on Mathematics of Data Science}~5 (2023), 829--1190.
\url{https://doi.org/10.1137/22M1498991}


\bibitem{BCLN}
K.\ Bredies, E.\ Chenchene, D.A.\ Lorenz, E.\ Naldi:
Degenerate preconditioned proximal point algorithms,
\emph{SIAM Journal on Optimization}~32 (2022), 2376--2401. 
\url{https://doi.org/10.1137/21M1448112}

\bibitem{BAC}
L.M.\ Brice\~{n}o-Arias, P.L.\ Combettes: 
A monotone+skew splitting model for composite monotone inclusions in duality, 
\emph{SIAM Journal on Optimization} (2011) 21, 1230--1250. 
\url{https://doi.org/10.1137/10081602X}


\bibitem{Paramonotone}
Y.\ Censor, A.N.\ Iusem, S.A.\ Zenios: 
An interior point method with Bregman functions for the variational inequality problem with paramonotone operators. 
\emph{Mathematical Programming}~81 (1998), 373--400. \url{https://doi.org/10.1007/BF01580089}

\bibitem{CP}
A.\ Chambolle, T.\ Pock:
A First-Order Primal-Dual Algorithm for Convex Problems with Applications to Imaging,
\emph{Journal of Mathematical Imaging and Vision}~40 (2011), 120--145. \url{https://doi.org/10.1007/s10851-010-0251-1}

\bibitem{CPsurvey}
A.\ Chambolle, T.\ Pock:
An introduction to continuous optimization in imaging, 
\emph{Acta Numerica}~25 (2016), 161--319.
\url{https://doi.org/10.1017/S096249291600009X}

\bibitem{theusual}
P.L.\ Combettes:
The geometry of monotone operator splitting methods, 
\emph{Acta Numerica}~33 (2024), 2024.
\url{https://doi.org/10.1017/S0962492923000065}

\bibitem{Condat}
L.\ Condat, D.\ Kitahara, A.\ Contreras, A.\ Hirabayashi:
Proximal splitting algorithms for convex optimization, 
\emph{SIAM Review}~65 (2023), 375--435. 
\url{https://doi.org/10.1137/20M1379344}


\bibitem{EF}
J.\ Eckstein,  M.\ Ferris:
Smooth methods of multipliers for complementarity problems.
\emph{Mathematical Programming}~86 (1999), 65--90. 
\url{https://doi.org/10.1007/s101079900076}

\bibitem{ELP} 
B.\ Evens, P.\ Latafat, P.\ Patrinos:
Convergence of the Chambolle-Pock algorithm 
in the absence of monotonicity, 
\url{https://doi.org/10.48550/arXiv.2312.06540}

\bibitem{EPLP}
B.\ Evens, P.\ Pas, P.\ Latafat, P.\ Patrinos:
Convergence of the preconditioned proximal point method and Douglas--Rachford splitting in the absence of monotonicity, 
\emph{Mathematical Programming (Series A)} (2025).
\url{https://doi.org/10.1007/s10107-024-02182-0}

\bibitem{IusPara}
A.N.\ Iusem:
On some properties of paramonotone operators,
\emph{Journal of Convex Analysis}~5 (1998), 269--278.
\url{https://www.heldermann.de/JCA/jca05.htm#jca052}

\bibitem{LP}
P.\ Latafat, P.\ Patrinos:
Asymmetric forward–backward–adjoint splitting for
solving monotone inclusions involving three operators,
\emph{Computational Optimization and Applications}~68 (2017), 
57--93. 
\url{https://doi.org/10.1007/s10589-017-9909-6}


\bibitem{LFLL}
C.\ Li, D.\ Fang, G.\ L\'opez, M.A.\ L\'opez: 
Stable and total Fenchel duality for convex 
optimization problems in locally convex spaces, 
\emph{SIAM Journal on Optimization}~20(4) (2009), 1032--1051. 
\url{https://doi.org/10.1137/080734352}

\bibitem{Pennanen}
T.\ Pennanen:
Dualization of Generalized Equations of Maximal Monotone Type,
\emph{SIAM Journal on Optimization}~10 (2000) 809--835.
\url{https://doi.org/10.1137/S1052623498340448}

\bibitem{Rockafellar}
R.T.\ Rockafellar: \emph{Convex Analysis}, Princeton University Press, 1970. 

\bibitem{Robinson}
S.M.\ Robinson:
Composition duality and maximal monotonicity,
\emph{Mathematical Programming}~85 (1999) 1--13.
\url{https://doi.org/10.1007/s101070050043}

\bibitem{RyuYin}
E.K.\ Ryu, W.\ Yin:
\emph{Large-Scale Convex Optimization}, 
Cambridge University Press, 2023. 
\url{https://doi.org/10.1017/9781009160865}

\bibitem{Tib}
R.J.\ Tibshirani:
Regression shrinkage and selection via the lasso, 
\emph{Journal of the Royal Statistical Society Series B}~58 
(1996), 267--288. 

\bibitem{Tibuni}
R.J.\ Tibshirani:
The lasso problem and uniqueness,
\emph{Electronic Journal on Statistics}~7(2013), 1456--1490. 
\url{https://doi.org/10.1214/13-EJS815} 


\end{thebibliography}
\end{document}